\documentclass[12pt]{article}

\usepackage{amsmath,amstext,amsopn,color,amsthm, amssymb}
\usepackage{ushort}
\usepackage{bbm}

\theoremstyle{plain}

 \newtheorem{theorem}{Theorem}[section] 
\newtheorem{lemma}[theorem]{Lemma} \newtheorem{proposition}[theorem]{Proposition}

 \theoremstyle{definition}
\newtheorem{definition}{Definition}

 \theoremstyle{remark}
\newtheorem*{remark*}{Remark} \newtheorem{remark}{Remark}

\numberwithin{equation}{section}

 \newcommand{\R}{\mathbb{R}}

\newcommand{\RR}{\R} \newcommand{\Rd}{{\R^{d}}}

 \newcommand{\I}{{J_1}}
\newcommand{\II}{{J_2}} \newcommand{\III}{{J_3}}
 
\renewcommand{\leq}{\leqslant} \renewcommand{\le}{\leq}
\renewcommand{\geq}{\geqslant} \renewcommand{\ge}{\geq}

\newcommand{\Cg}{C_0}

\DeclareMathOperator{\dist}{dist}
\DeclareMathOperator{\diam}{diam}

 \def\({\left(} \def\){\right)} \def\[{\left[}
  \def\]{\right]} \def\<{\langle} \def\>{\rangle}

\ushortCreate*(3)(.46){uline}

\newcommand{\WUSC}[3]{\textrm{\rm WUSC}(#1,#2,#3)}
\newcommand{\WLSC}[3]{\textrm{\rm WLSC}(#1,#2,#3)}
\newcommand{\lC}{{\underline{c}}}
\newcommand{\uC}{{\overline{C}}}
\newcommand{\la}{{\underline{\alpha}}}
\newcommand{\ua}{{\overline{\alpha}}}
\newcommand{\lt}{{\underline{\theta}}}
\newcommand{\ut}{{\overline{\theta}}}

\newcommand{\tX}{\tilde{X}}

\newcommand{\tnu}{\tilde{\nu}}
\newcommand{\tp}{\tilde{p}}
\newcommand{\tG}{\tilde{G}}
\newcommand{\tr}{\tilde{r}}

\def \EE {\mathbb{E}} \def \PP {\mathbb{P}} \def \RR {\mathbb{R}}  \def \Rd {{\RR^d}}     
\def \pK {\mathcal{K}^\nabla}

       \def \CII {C_2} \def \CIII
{C_3} \def \CIV {C_5}  \def \CVI {C_6}

\def \CXX{C_{1}}
\def \CXXI{C_{4}}

\def \kk {G}
\def \tp{\tilde{p}}
\def \tEE {\tilde{\mathbb{E}}}
\def \tPP {\tilde{\mathbb{P}}}
\def \PP {\mathbb{P}}

\def \tu {\tilde{u}}

\def \tG {\tilde G}
\def \tL {\tilde{\mathcal{L}}}

\def \tGDD {{\tilde G_D}}
\def \tP {\tilde P}

\def\deltaDD {\delta_D}

\title{
Green function for gradient perturbation of
 unimodal L\'evy processes
\footnotetext{The research was  partially supported by grant MNiSW IP2012 018472 and the Alexander von Humboldt Foundation.
}}

\author{Tomasz  Grzywny, Tomasz Jakubowski and Grzegorz \.{Z}urek\\
  Wroc\l{}aw University of Technology, Poland}

\begin{document}
\maketitle
\begin{abstract}
We prove that the Green function of a generator of isotropic unimodal L\'{e}vy processes with the weak lower scaling order
bigger than one and the Green function of its gradient perturbations are comparable
for bounded smooth  open sets if the drift function is from  an appropriate Kato class.
\end{abstract}

\section{Introduction}

Let $X_t$ be a pure-jump isotropic  unimodal L\'{e}vy process on $\Rd$, $d\geq 2$. That is, $X_t$ is a L\'evy processes with a rotationally invariant and radially non-increasing density function $p_t(x)$ on $\R^d \setminus \{0\}$. The characteristic exponent of $\{X_t\}$ equals
$$
\psi(x) = \int_{\RR^d}\left (1 -\cos ( x\cdot  z )\right) \nu(dz), \,\quad x\in\Rd.
$$
where   $\nu$ is a L\'{e}vy measure, i.e., $\int_{\Rd}\left(1\wedge|z|^2\right)\nu(dz)<\infty$.
 For general information on unimodal processes, we refer the reader to \cite{MR3165234, MR3225805, MR705619}.
One of the primary example of a mentioned class of processes is the isotropic $\alpha$-stable L\'evy process having the fractional Laplacian $\Delta^{\alpha/2}$ as a generator.

Perturbations of $\Delta^{\alpha/2}$ by the first order operators are currently widely studied by many authors \cite{MR2283957, MR2892584, MR3238505, MR2680400, MR3050510,  MR3129851, MR2369047, MR2875353,  MR2276260, MR3035054, MR3283160, MR1310558, MR3060702, MR3017289} from various points of view. In a recent paper \cite{MR2892584} the authors studied the Green function of $\Delta^{\alpha/2} + b(x) \cdot \nabla$ in bounded $C^{1,1}$ domains. Here $b$ is a vector field from the Kato class $K_d^{\alpha-1}$. It was shown that the Green function of the original process is comparable with the Green function of the perturbed process. In this paper we generalize the result of \cite{MR2892584} to the case of isotropic unimodal L\'evy processes. Let
\begin{equation}\label{generator}
{\cal L} f(x)= \int_{\Rd} \left(f(x+z)-f(x)
-{\bf 1}_{|z|<1}(z\cdot \nabla f(x))\right)\nu(dz)\,, \quad f\in C^2_b(\R^d)\,,
\end{equation}
be a generator of the process $X_t$. We will consider a non-empty bounded open $C^{1,1}$ set $D$ and the Green function $G_D$ for $\cal{L}$. Now, let $\tilde G_D(x,y)$ be a Green function for
$$
\tilde{\mathcal L} = \mathcal{L} + b(x) \cdot \nabla,
$$
where $b$ is a function from the Kato class  $\pK_d$ (see Section 2 for details). Our main result is


\begin{theorem}
\label{Theorem1}
Let $d\geq 2$, $b\in \pK_d$, and  let $D\subset \Rd$
be a bounded $C^{1,1}$ open set. We assume that the characteristic exponent
$$\psi\in \WLSC{\la}{0}{\lC}\cap \WLSC{\la_1}{1}{\lC_1} \cap \WUSC{\ua}{0}{\uC}, \qquad  \mbox{where }\la_1>1,$$
\begin{equation}
\label{gradEst}|\nabla_x G_D(x,y)| \leq \Cg \frac{ G_D(x,y)}{|x-y|\wedge\delta_D(x)\wedge 1}.
\end{equation}
Then, there exists a constant $C$ such that for $x,y \in D$,
  \begin{equation}
    \label{eq:egf}
C^{-1}G_D(x,y) \le \tilde G_D(x,y) \le C G_D(x,y)\,.
  \end{equation}
\end{theorem}
Here WLSC and WUSC are the classes of functions satisfying a weak lower and weak upper scaling condition, respectively (see Section 2). The condition (\ref{gradEst}) is satisfied for a wide class of processes. For example, (\ref{gradEst}) holds under a mild assumption on a density of the L\'{e}vy measure, which is satisfied for any subordinate Brownion motion (see Lemma \ref{ex:gradEst}), (see also \cite[Theorem 1.4]{2015arXiv150102023C}).

Generally, we follow the approach of \cite{MR2892584}. Since some proofs are almost identical to the ones from \cite{MR2892584}, we omitted them. The main tool, we use in this paper, is the Duhamel (perturbation) formula (see Theorem \ref{lem:pf}). We note that this result cannot be obtained directly in the same way as the perturbation formula for fractional Laplacian (see \cite[Lemma 12]{MR2892584}. One of the other difficulties in this paper is that we do not have the explicit formula for the potential kernel $U(x)$ of $X_t$. Moreover, for stable process $\psi(\xi) = |\xi|^\alpha$, which gives a nice scaling of some main objects. Here, we have only weak scaling but it is sufficient for our purpose, although it makes the calculations a little harder. For example, in the estimates of the Green function a factor $V(\delta_D(x))$ appears. For stable process $V(r) = r^{\alpha/2}$ and if $y$ is such that $\delta_D(y) = \lambda \delta_D(x)$, then $V(\delta_D(y)) = \lambda^{\alpha/2} V(\delta_D(x))$. For the general unimodal process, $V$ satisfies weak scaling condition and we can only estimate $V(\delta_D(y))$.

The paper is organized as follows. In Section 2 we give the definitions of the processes $X$ and $\tilde{X}$ and present their basic properties. In Section 3, we introduce Green functions of $X$ and $\tilde{X}$. Lastly, in Section 4, we prove Theorem \ref{Theorem1}.

When we write $f (x) \approx g(x)$, we mean that there is a number $0 < C < \infty$ independent of $x$, i.e.
a constant, such that for every $x$ we have $C^{−1} f (x) \le g(x) \le C f (x)$. The notation
$C = C(a, b, \ldots, c)$ means that $C$ is a constant which depends only on $a, b, \ldots , c$.
We use a convention that constants denoted by capital letters do not change throughout the paper. For a radial function  $f:\Rd\rightarrow [0,\infty)$ we shall often write $f(r)=f(x)$ for any $x \in \Rd$ with $|x| = r$.

\section{Preliminaries}
In what follows, $\RR^d$ denotes the Euclidean space of dimension $d\ge
2$, $dy$ stands for the Lebesgue measure on $\Rd$. Without further mention we will only consider Borelian sets, measures and functions in $\RR^d$. As usual, we write $a \land b = \min(a,b)$ and $a \vee b = \max(a,b)$. By $x\cdot y$ we denote the Euclidean scalar product of $x,y\in \Rd$.
We let $B(x,r)=\{y\in \RR^d \colon |x-y|<r\}$.
For $D\subset \RR^d$, the distance to the complement of $D$, will be denoted by
$$\delta_D(x) =\dist(x,D^c)\,.$$

\begin{definition}
Let $\lt\in [0,\infty)$ and
 $\phi$ be a non-negative non-zero  function on $(0,\infty)$.
We say that
$\phi$ satisfies {the} {\it \bfseries weak lower scaling condition} (at infinity) if there are numbers
$\la>0$ and  $\lC \in(0,1]$  such that
\begin{equation}\label{eq:LSC}
 \phi(\lambda\theta)\ge
\lC\lambda^{\,\la} \phi(\theta)\quad \mbox{for}\quad \lambda\ge 1, \quad\theta>\lt.
\end{equation}

In short, we say that $\phi$ satisfies WLSC($\la, \lt,{\lC}$) and write $\phi\in\WLSC{\la}{ \lt}{\lC}$.
If $\phi\in\WLSC{\la}{0}{\lC}$, then we say
that $\phi$ satisfies the {\bfseries  \emph{global} weak lower scaling condition}.

Similarly, we consider
 $\ut\in [0,\infty)$.
The  {\it \bfseries weak upper scaling condition} holds if there are numbers $\ua<2$
and $\uC{\in [1,\infty)}$ such that
\begin{equation}\label{eq:USC}
 \phi(\lambda\theta)\le
\uC\lambda^{\,\ua} \phi(\theta)\quad \mbox{for}\quad \lambda\ge 1, \quad\theta> \ut.
\end{equation}
In short, $\phi\in\WUSC{\ua}{ \ut}{\uC}$. For {\bfseries  \emph{global} weak upper scaling} we require $\ut=0$ in \eqref{eq:USC}.
\end{definition}

Throughout the paper, $X_t$ will be the pure-jump isotropic unimodal L\'evy process on $\RR^d$. The L\'evy measure $\nu$ of $X_t$ is radially symmetric and non-increasing, so it admits the radial density $\nu$, i.e., $\nu(dx) = \nu(|x|)dx$.   Hence  the characteristic exponent $\psi$ of $X_t$ is radial as well.   
We assume that (see Theorem \ref{Theorem1})
\begin{align}
\psi  & \in \WLSC{\la}{0}{\lC} \cap \WUSC{\ua}{0}{\uC}\,, \label{eq:assum1}\\
\psi  & \in  \WLSC{\la_1}{1}{\lC_1}, \qquad  \mbox{for some } \la_1>1\,. \label{eq:assum2}
\end{align}

\noindent Following \cite{MR632968}, we define
$$
h(r)=\int_\Rd \(1\wedge \frac{|x|^2}{r^2}\)\nu(|x|)dx, \qquad r>0\,.
$$
Let us notice that
$$
h(\lambda r)\leq h(r)\leq \lambda^2  h(\lambda r), \quad \lambda>1.
$$
Moreover, by \cite[Lemma 1 and (6)]{MR3165234}
$$
2^{-1}\psi(1/r)\leq h(r)\leq \CXX \psi(1/r).
$$
In fact, we may write $\CXX = d\pi^2/2$ but it will be more convenient to write this constant as $\CXX$.
We define the function $V$ as follows,	
$$V(0)=0 \,\,\,\mathrm{and}\,\,\, V(r)=1/\sqrt{h(r)}, \quad r>0.$$
Since $h(r)$ is non-increasing, $V$ is non-decreasing. We have
\begin{equation}\label{subadd} V(r)\leq V(\lambda r)\leq \lambda V(r),\quad r\geq 0 ,\,\lambda >1.
\end{equation}
By weak scaling properties of $\psi$ and the property $h(r) \approx \psi(1/r)$, we get
\begin{equation}\label{scalV1} \(\frac{\lC}{2\CXX}\)^{1/2}\lambda^{\la/2}\leq \frac{V(\lambda r)}{V(r)}\leq (2\uC  \CXX)^{1/2}\lambda^{\ua/2},\quad r> 0 ,\,\lambda >1,
\end{equation}
\begin{equation}\label{scalV2}
\frac{V(\eta r)}{V(r)} \leq \(\frac{2\CXX}{\lC_1}\)^{1/2}\eta^{\la_1/2},\quad \,\eta<1,   r<1 .
\end{equation}
Therefore,
$ V   \in \WLSC{\la/2}{0}{\sqrt{\lC/(2\CXX)}} \cap \WUSC{\ua/2}{0}{\sqrt{2\uC\CXX}}$.\\

\begin{remark}\label{remScalExp}
The threshold $(0,1)$ in scaling of $V$ in \eqref{scalV2} may be replaced by any bounded interval at the expense of constant $\sqrt{2\CXX/\lC_1}$ (see  \cite[Section 3]{MR3165234}), i.e., for any $R>1$, there is a constant $c$ such that
\begin{equation}\label{scalV2R}
\frac{V(\eta r)}{V(r)} \leq \ c \eta^{\la_1/2},\quad \,\eta<1,   r<R .
\end{equation}
\end{remark}

The  global weak  lower scaling condition (assumption (\ref{eq:assum1})) implies $p_t(x)$ is jointly continuous on $(0,\infty)\times \Rd$ ($e^{-t\psi}\in L^1(\Rd)$) and (see \cite[Lemma 1.5]{MR3249349})
\begin{eqnarray}
p_t(x)&\approx& [V^{-1}(\sqrt{t})]^{-d}\wedge \frac{t}{V^2(|x|)|x|^d}   \label{p_t_L}\\
\nu(x)&\approx&\frac{1}{V^2(|x|)|x|^d}. \label{nu_L}
\end{eqnarray}


Analogously to $\alpha$-stable processes, we define the Kato class for gradient perturbations.
\begin{definition}
We say that a vector field $b \colon \RR^d \to \RR^d$ belongs to the Kato class $\pK_d$  if
\begin{equation}\label{eq:Kc}\lim_{r\to0^+}\sup_{x\in\Rd}\int_{B(x,r)}\frac{V^2(|x-z|)}{|x-z|^{d+1}} |b(z)|dz=0.\end{equation}
\end{definition}
\begin{remark}
We note that $L^{\infty}(\Rd)\subset \pK_d$.
\end{remark}

\noindent Let us denote
$$
p(t,x,y)=p_t(y-x)\,.
$$
By \cite[Theorem 3.4]{GS}, we have
$$|\nabla_x p(t,x,y)|\leq c \frac{1}{V^{-1}(\sqrt{t})} p(t,x,y), \quad t>0, x,y\in\Rd.$$
Let $b \in \pK_d$. Following \cite{MR2283957} and \cite{MR2876511}, for $t>0$ and $x,y \in \RR^d$,  we recursively define
$$p_0(t,x,y)  =  p(t,x,y)\,,$$
$$
 p_n(t,x,y)  =  \int_0^t \int_{\RR^d} p_{n-1}(t-s,x,z) b(z) \cdot \nabla_z p(s,z,y)\,dz\,ds\,,\quad n \ge 1\,,
$$
and we let
\begin{equation}
\tp=\sum_{n=0}^\infty p_n\,.
\end{equation}
By \cite[Theorem 1.1]{MR2876511}, the series converges to a probability transition density function,  and
\begin{equation}\label{ptxy_comp}
  c_T^{-1} p(t,x,y) \le \tp(t,x,y) \le c_Tp(t,x,y)\,,\qquad x,y \in \RR^d\,,\; 0<t<T\,,
\end{equation}
where $c_T \to 1$ if $T \to 0$, see \cite[Theorem 3]{MR2876511}. Moreover, one can prove that $\tp$ is jointly continuous on $(0,\infty)\times\Rd\times\Rd$ (see \cite[Corollary 19]{MR2283957}). 

%

We consider the time-homogeneous transition probabilities
$$
P_t(x,A) =\int_A p(t, x,y)dy\,,\qquad \tilde{P}_t(x,A) =\int_A \tp(t, x,y)dy,
$$
$t>0$, $x\in \Rd$, $A\subset\Rd$.
By Kolmogorov's and Dinkin-Kinney's theorems the transition
probabilities $P_t$ and $\tilde{P}_t$ define in the usual way Markov probability measures
$\{\PP^x, \tPP^x,\,x\in \Rd\}$ on the space $\Omega$ of the
right-continuous and left-limited functions $\omega :[0,\infty)\to \Rd$.
We let $\EE^x, \tEE^x$ be the corresponding expectations.
We will denote by $X=\{X_t\}_{t\geq 0}$  the canonical process on $\Omega$, $X_t(\omega) = \omega(t)$. Hence,
$$
\PP(X_t \in B) = \int_B p(t,x,y)dy,\qquad \tPP(X_t \in B) = \int_B \tp(t,x,y)dy
$$


For any open set $D$ we define {\it the first exit time}\/ of the process $X_t$ from $D$,
$$\tau_D=\inf\{t>0: \, X_t\notin D\}\,.$$
Now, by the usual Hunt's formula, we define the transition density of the process {\it killed}\/ when leaving $D$
(\cite{MR0264757}, \cite{MR1329992}, \cite{MR3249349}):
$$
p_D(t,x,y)=p(t,x,y)-\EE^x[\tau_D<t;\, p(t-\tau_D, X_{\tau_D},y)],\quad t>0
,\,x,y\in \Rd \,.
$$
We briefly recall some well known properties of $p_D$ (see \cite{MR3249349}).
The function $p_D$ satisfies the Chapman-Kolmogorov equations
$$
\int_\Rd p_D(s,x,z)p_D(t,z,y)dz=p_D(s+t,x,y)\,,\quad s,t>0 ,\,
x,y\in \Rd\,.
$$
Furthermore, $p_D$ is jointly continuous (compare Lemma \ref{lem:DhkCon}) when $t\neq 0$, and we have
\begin{equation}\label{eq:gg}
  0\leq p_D(t,x,y)=p_D(t,y,x)\leq p(t,x,y)\,.
\end{equation}
In particular,
\begin{equation}
  \label{eq:9.5}
  \int_\Rd p_D(t,x,y)dy\leq 1\,.
\end{equation}
By Blumenthal's 0-1 law, radial symmetry of $p_t$
and $C^{1,1}$ geometry of the boundary of $\partial D$, we have
$\PP^x(\tau_D=0)=1$ for every $x\in D^c$.
In particular, $p_D(t,x,y)=0$ if $x\in D^c$ or $y\in D^c$.
By the strong Markov property,
$$
\EE^x[t<\tau_D;\, f(X_t)]=
\int_\Rd f(y)p_D(t,x,y)dy
\,,
\quad t>0\,,\; x\in\Rd\,,
$$
for functions $f\geq 0$.

For $s\in \RR$, $x\in \Rd$, and $\phi\in C^\infty_c(\RR\times D)$, we
have (see \cite[Remark 4.2]{2014arXiv1411.7907B} and \cite[the proof of Lemma 5]{MR2892584}) 
\begin{equation}
  \label{eq:fsolD}
  \int_{s}^\infty\int_{D}
  p_D(u-s,x,z)\left[
    \partial_u\phi(u,z)+\mathcal{L}_z \phi(u,z)\right]\,dzdu
  = -\phi(s,x)\,,
\end{equation}
which justifies calling $p_D$ the Dirichlet heat kernel of $\mathcal{L}$ on  $D$.

In a similar way, we define analogous object for the process $\tX$. Let $\tilde\tau_D=\inf\{t>0: \, \tX_t\notin D\}$.
 By Hunt's formula,
\begin{equation}\label{eq:Hunt}
\tp_D(t,x,y) = \tp(t,x,y) - \tEE^x\left[\tau_D < t;\; \tp(t-\tau_D,  X_{\tau_D},y)\right]\,.
\end{equation}
Except symmetry, $\tp_D$ has analogous properties as $p_D$, i.e. the Chapman-Kolmogorov equation holds
$$
\int_{\RR^d} \tp_D(s,x,z)\tp_D(t,z,y)dz=\tp_D(s+t,x,y)\,,\quad s,t>0 ,\,
x,y\in \Rd
$$
and $0\leq \tp_D(t,x,y)\leq \tp(t,x,y)$. Now, we will prove that
$\tp_D$ is jointly continuous on $(0,\infty)\times D \times D$. First, we need two preparatory lemmas.

\begin{lemma}\label{sup-p-t1}
Let $\delta>0$. Then $M_\delta:=\sup_{t>0,|x-y|\geq \delta}\tp(t,x,y)<\infty$.
\end{lemma}
\begin{proof}
By \eqref{ptxy_comp} and \cite[Corollary 7]{MR3165234}, for $t\leq 1$ 
$$\tp(t,x,y)\leq c \frac{t}{V^2(|x-y|)|x-y|^d}.$$
Hence,
$$\sup_{0<t\leq 1,|x-y|\geq \delta}\tp(t,x,y)\leq \frac{c}{V^2(\delta)\delta^{d}}.$$
Furthermore,  by the semigroup property, for $t>1$,
$$\tp(t,x,y)\leq c\int_{\Rd}\tp(t-1,x,z)p(1,z-y)dy\leq c p(1,0).$$
These imply $$M_\delta\leq c\max\{(V(\delta)\delta^{d})^{-1},p(1,0)\}<\infty.$$
\end{proof}

\begin{lemma}
Let $\delta>0$. Then
\begin{eqnarray}
\lim_{s\to 0^+}\sup_{t\leq s, x\in \Rd} \tPP^x(|X_t-X_0|\geq \delta)&=&0,\label{tailPr}\\
\lim_{s\to 0^+}\sup_{x\in \Rd} \tPP^x(\tau_{B(x,\delta)}\leq s)&=&0.\label{exitPr}
\end{eqnarray}
\end{lemma}
\begin{proof}
Let $s\leq 1$ and $t\leq s$. By \eqref{ptxy_comp} and \cite[Corollary 6]{MR3165234},
$$\tPP^x(|X_t-X_0|\geq \delta)\leq c_1\int_{B_{\delta}^c}p(t,y)dy\leq c \frac{t}{V^2(\delta)}\leq c(\delta)s.$$
Hence, we obtain \eqref{tailPr}.  \eqref{exitPr}  is a consequence of \eqref{tailPr} and the strong Markov property (see \cite[the proof of Lemma 3.1]{MR3050510}).
\end{proof}

Although, in this paper, we consider only bounded sets, the following lemma also holds for unbounded domains. To obtain it we use standard arguments (e.g.,\cite[Theorem 2.4]{MR1329992}).
\begin{lemma}\label{lem:DhkCon}
$\tp_D$ is jointly continuous on $(0,\infty)\times D \times D$.
\end{lemma}
\begin{proof}Let $0<\delta<r$,  $D^\delta=\{y\in D: \delta_D(y)\geq \delta \}$ and $D_r^\delta=D^\delta\cap B_r$. Generally, $\delta$ is close to $0$ and $r$ is large. 
We assume that $(t,x,y)\in [\delta,r]\times D^\delta \times D_r^\delta$.
We denote by
$$
\tilde{r}_D(t,x,y) = \tEE^{x}[\tp(t-\tau_D,X_{\tau_D},y),\tau_D<t],
$$
the killing measure of $\tX$. Hence,
$$
\tp_D(t,x,y) = \tp(t,x,y) - \tilde{r}_D(t,x,y).
$$
Let $s<\delta/2$,

$$h_s(t,x,y)= \tEE^{x}[\tp(t-s-\tau_D,X_{\tau_D},y),\tau_D<t-s]$$ and $\phi_s(t,x,y)=\tEE^x h_s(t,X_s,y)$.
By Markov property,
\begin{align*}
\tr_D(t,x,y)&-\phi_s(t,x,y)=\tEE^x[\tp(t-\tau_D,X_{\tau_D},y),\tau_D\leq s]\\&-\tEE^x[\tau_D\leq s,\tEE^{X_s}[\tp(t-s-\tau_D,X_{\tau_D},y),\tau_D<t-s]]
\end{align*}
By Lemma \ref{sup-p-t1}, \begin{equation}\label{con-p-t-1}|\tr_D(t,x,y)-\phi_s(t,x,y)|\leq 2M_{\delta}\tPP^x(\tau_D\leq s)\leq 2M_{\delta}\sup_{z\in\Rd}\tPP^z (\tau_{B(z,\delta)}\leq s).\end{equation}
Hence, by \eqref{exitPr}, it is enough to prove continuity of $\phi_s$ on $[\delta,r]\times D^\delta \times D_r^\delta$ for $0<s<\delta/2$.

 First, we prove equicontinuity of $h_s(\cdot,z,\cdot)$ on $[\delta,r]\times D_r^\delta$ for $z\in\Rd$. Fix $\varepsilon>0$. By \eqref{ptxy_comp} and (\ref{p_t_L}), there is $0<\lambda\leq\delta/4 $ such that  for $w\in D^c$, $v\in D^\delta$ and  $u\leq \lambda$,
\begin{equation}\label{eq:cont-p1}\tp(u,w,v)\leq c \frac{\lambda}{V^2(\delta)\delta^{d}}<\varepsilon.
\end{equation}
Next, by the semigroup property, (\ref{p_t_L}) and \eqref{ptxy_comp}, there is $R\geq 2r$ such that for $w\in B_R^c$, $v\in B_r$ and $u\leq r$,
\begin{equation}\label{eq:cont-p2}\tp(u,w,v)\leq \frac{c r}{ V^2(R/2)R^{d}}<\varepsilon.
\end{equation}
Now, we divide $h_s$ into tree parts and we treat them separately,
$$h_s(t,z,y)=\I(t,z,y)+ \II(t,z,y)+\III(t,z,y),$$
where
\begin{eqnarray*}\I(t,z,y)&=&\tEE^z[\tp(t-s-\tau_D,X_{\tau_D},y),\tau_D<t-s-\lambda, X_{\tau_D}\in B_{R}]\\
\II(t,z,y)&=&\tEE^z[\tp(t-s-\tau_D,X_{\tau_D},y),t-s-\lambda\leq\tau_D<t-s]\\ \III(t,z,y)&=& \tEE^z[\tp(t-s-\tau_D,X_{\tau_D},y),\tau_D<t-s-\lambda, X_{\tau_D}\in B^c_{R}].
\end{eqnarray*}
By \eqref{eq:cont-p1} and \eqref{eq:cont-p2},
\begin{equation}\label{eq:cont-p3}
\II(t,z,y)+\III(t,z,y)<2\varepsilon,\quad z\in \Rd, (t,y)\in[\delta,r]\times D_r^\delta.
\end{equation}
Since $\tp(\cdot,\cdot,\cdot)$ is continuous on $(0,\infty)\times\Rd\times\Rd$, it is uniform continuous on $[\lambda/2, r]\times B_{R}\times B_r$.  Hence, there is $0<\varepsilon_1\leq \lambda/2$  such that for $(u,w),(u_0,w_0)\in [\lambda/2,r]\times D_r^\delta$ and $w\in B_R$,
\begin{equation}\label{eq:cont-p4}|\tp(u,v,w)-\tp(u_0,v,w_0)|<\varepsilon, \quad \mathrm{if }\,\,\,\, |(u,w)-(u_0,w_0)|<\varepsilon_1, \; v \in \RR^d. \end{equation}
Assume that $(t_0,y_0)\in [\delta,r]\times D_r^\delta$ and $t_0\leq t$. Then
\begin{eqnarray*}\I(t_0,z,y_0)&=&\tEE^z[\tp(t_0-s-\tau_D,X_{\tau_D},y_0)|,\tau_D<t-s-\lambda, X_{\tau_D}\in B_{R}]\\
&-&\tEE^z[\tp(t_0-s-\tau_D,X_{\tau_D},y_0),t_0\leq\tau_D+s+\lambda<t, X_{\tau_D}\in B_{R}].
\end{eqnarray*}
This, \eqref{eq:cont-p1} and  \eqref{eq:cont-p4} imply, for $(t,y),(t_0,y_0)\in [\delta,r]\times D_r^\delta$,
\begin{equation}\label{eq:cont-p5}\sup_{z\in \Rd}|\I(t,z,y)-\I(t_0,z,y_0)|< 2 \varepsilon,\quad \mathrm{if} \,\,\,|(t,y)-(t_0,y_0)|<\varepsilon_1, \; z \in \RR^d.
\end{equation}
Combining  \eqref{eq:cont-p3} with \eqref{eq:cont-p5} gives equicontinuity of $h_s(\cdot,z,\cdot)$ on $[\delta,r]\times D_r^\delta$ for $z\in\Rd$.

This implies equicontinuity of $\phi_s(\cdot,z,\cdot)$ on $[\delta,r]\times D_r^\delta$ for $z\in\Rd$.
Since $\tilde{P}_s$ is strong Feller, $\phi_s(t,\cdot,y)$ is continuous on $\Rd$. Therefore $\phi_s(\cdot,\cdot,\cdot)$ is jointly continuous on $[\delta,r]\times\Rd\times D_r^\delta$. By \eqref{con-p-t-1} and \eqref{exitPr}, $r_D(\cdot,\cdot,\cdot)$ is jointly continuous on  $[\delta,r]\times D_r^\delta\times D_r^\delta$, what implies continuity on $(0,\infty)\times D \times D$. Since $\tp$ is jointly continuous, $\tp_D$ is jointly continuous on $(0,\infty)\times D \times D$.
\end{proof}


By similar calculations like \cite[Theorem 2]{MR2876511}, one can prove that $\tp$ is the fundamental solutions for $\tL$.
\begin{lemma}\label{l:wkp}
For $s>0$, $x\in D$ and $\phi \in C_c^\infty\big((0,\infty)\times D\big)$, we have

\begin{equation}\label{tp_D_fs}
\int_s^\infty \int_{D} \tp_D(u-s,x,z) \left(\partial_u + \tilde{\mathcal{L}}  \right) \phi\,(u,z) \,dz\,du = - \phi(s,x)\,.
\end{equation}
\end{lemma}

\section{Green functions}
In this section we define and prove some properties of the Green functions of $\mathcal{L}$ and $\tilde{\mathcal{L}}$.
\subsection{Green function of $\mathcal{L}$}

\begin{definition}\label{def:C11}
 Non-empty open $D\subset \Rd$ is of class $C^{1,1}$ at scale $r>0$
if for
  every $Q\in \partial D$ there are balls
$B(x',r)\subset D$ and $B(x'',r)\subset D^c$ tangent at $Q$.
\end{definition}
If $D$ is $C^{1,1}$ at some unspecified scale (hence also at all smaller scales),  then we simply say $D$ is $C^{1,1}$.
The {\it localization radius}, $$r_0=r_0(D)=\sup\{r: D \mbox{ is } C^{1,1} \mbox{ at scale } r\},$$
refers to the local geometry of $D$, while the {\it diameter},
$${\rm diam}(D)=\sup\{|x-y|:\;x,y\in D\}\,,$$ refers to the global geometry of $D$.
The ratio ${\rm diam}(D)/r_0(D)\geq 2$ will be called the {\it distortion}\/ of $D$.
We can localize each $C^{1,1}$ open set as follows (see \cite[Lemma 1]{MR2283957})
\begin{lemma}\label{l:loc}
There exists $\kappa>0$ such that if $D$ is $C^{1,1}$ at scale $r$ and $Q\in \partial D$, then there is a $C^{1,1}$ domain $F\subset D$ with $r_0(F)>\kappa r$, ${\rm diam }(F)< 2r$ and
\begin{equation}\label{eq:loc}
D\cap B(Q,r/4)=F\cap B(Q,r/4)\,.
\end{equation}
\end{lemma}
\noindent
We will write $F=F(z,r)$, and we note that the distortion of $F$ is at most $2/\kappa$, an absolute constant.
$$\mbox{In what follows $D$  will be a non-empty bounded $C^{1,1}$ open set in $\Rd$.}$$
We note that such $D$ may be disconnected but then it may only have a finite number of connected components, at a positive distance from each other.

\begin{definition}
We say that a function $h$ is $\mathcal{L}$-harmonic in the open set $D$ if for every $U$ such that $\overline{U} \subset D$, we have
$$
h(x) = \EE^x h(X_{\tau_U}), \quad x \in \RR^d\,.
$$
\end{definition}
\noindent We define the Green function of $\mathcal{L}$ for $D$,
\begin{align}\label{eq:12.5}
 G_D(x,y)=\int_0^\infty p_D(t,x,y)dt,  \qquad x,y \in \RR^d\,.
\end{align}
We briefly recall some basic properties of $G_D(x,y)$ (see \cite{{MR3249349}} for details). For $x\in D^c$ or $y\in D^c$, $G_D(x,y) = 0$. $G_D(x,y)$ is symmetric, continuous for $x\neq y$ and ${G_D}(x,x) = \infty$ for $x\in D$.
Furthermore, $G_D(\cdot,y)$ is $\mathcal{L}$-harmonic in $D \setminus \{y\}$ for every $y \in D$. We also have

\begin{lemma}\label{ex:gradEst}
Let $-\frac{\nu'(r)}{r}$ be non-increasing. Then,	 \eqref{gradEst} holds.
\end{lemma}
\begin{proof}
	Since $G_D(\cdot,y)$ is $\mathcal{L}$-harmonic on $D\setminus B(y,r)$ for small $r>0$, by \eqref{nu_L} and  \cite[Theorem 1.1, Proposition 1.3]{TKMR}, we have
\begin{equation}\label{eq:estGradGreena}
|\nabla_x G_D(x,y)| \leq c\frac{ G_D(x,y)}{\delta_{D\setminus B(y,\frac{|x-y|}{2})}(x)} \leq 2c\frac{ G_D(x,y)}{|x-y|\wedge\delta_D(x)\wedge 1}.
\end{equation}
\end{proof}

The Green operator of $\mathcal{L}$ for $D$ is
$$
G_Df(x) = \EE^x \int_0^{\tau_D}f(X_t)dt=\int_{\Rd}G_D(x,y)f(y)dy, \quad
x \in \RR^d\,,
$$
and we have
\begin{equation}\label{eq:fg}
G_D (\mathcal{L}\phi) (x) = \int_D G_D(x,y) \mathcal{L}\phi(y) dy = -\phi(x)\,,\quad x\in \Rd,\, \phi \in C^\infty_c(D)\,.
\end{equation}
By Ikeda - Watanabe formula \cite{MR0142153},
the $\PP^x$-distribution of
$X_{\tau_D}$ has a density function, called the Poisson kernel and defined as
\begin{equation}
  \label{eq:Pk}
  P_D(x,z)=\int_D G_D(x,y)\nu(z-y)dy\,, \qquad x \in D,\; z \in (\overline{D})^c\,.
\end{equation}
Hence,
$$
\PP^x(X_{\tau_D} \in B) = \int_B P_D(x,z) dz\,, \qquad B \subset (\overline{D})^c\,.
$$
Because of the $C^{1,1}$ geometry of $D$, $\PP^x(X_{\tau_D} \in \partial D)=0$ (\cite{MR1825650}), hence, the above formula holds for $B \subset D^c$ (we put $P_D(x,z) = 0$ for $z \in \partial D$).

By $G$ we denote the potential kernel of $X$ , that is
$$
G(x) = \int_0^\infty p_t(x)\,dt\,,
$$
which is finite on $\Rd\setminus\{0\}$ since $d\geq2$ and the global weak upper scaling condition for $\psi$ holds.
For $x \in \RR^d \setminus \{0\}$, we denote
$$
U(x)=\frac{V^2(|x|)}{|x|^d}.
$$
We note that by \eqref{subadd},  $U(x)$ is radially non-increasing. In \cite[Theorem 3 and Section 4]{MR3225805}, it was proved that $G(x)\approx U(x)$ for $x\neq 0.$
Let
$$r(y,z)=\delta_D(y)\vee \delta_D(z)\vee |y-z|\,.$$
\begin{lemma}\label{GreenEst1}
Let $D$ be a bounded open $C^{1,1}$ set. Then
$$G_D(y,z)\approx U(y-z)\frac{V(\delta_D(y))V(\delta_D(z))}{V^2(r(y,z))}, \quad y,z\in \Rd, $$
where the comparability constant depends only on $\psi$ and a distortion of $D$.
\end{lemma}
\begin{proof}
Taking the estimates of $p_D(t,x,y)$ (see \cite[Proposition 4.4 and  Theorem 4.5]{MR3249349}) and integrating them against time (see
\cite[the proof of Theorem 7.3]{2013arXiv1303.6449C}), we get
$$G_D(y,z)\approx U(y-z)\(\frac{V(\delta_D(y))V(\delta_D(z))}{V^2(|y-z|)}\wedge 1\),$$
where the comparability constant depends on $\psi$ only through the scaling characteristics	 and a distortion of $D$.
Since $V$ is non-decreasing, we have
$$\frac{V(\delta_D(y))V(\delta_D(z))}{V^2(r(y,z))} \leq \frac{V(\delta_D(y))V(\delta_D(z))}{V^2(|y-z|)}\wedge 1.$$
By symmetry of $G_D(x,y)$, we may assume that $\delta_D(y)\leq \delta_D(z)$. If $r(y,z)=|y-z|$, then
$$
\frac{V(\delta_D(y))V(\delta_D(z))}{V^2(|y-z|)}\wedge 1 = \frac{V(\delta_D(y))V(\delta_D(z))}{V^2(r(y,z))}.
$$
Let $r(y,z)=\delta_D(z)$. If $\delta_D(y) \geq \delta_D(z)/2$, then
$$\frac{V(\delta_D(y))V(\delta_D(z))}{V^2(r(y,z))} \ge \frac{V(\delta_D(y))}{V(2\delta_D(y))}\geq \frac{1}{2}\geq \frac{1}{2}\( \frac{V(\delta_D(y))V(\delta_D(z))}{V^2(|y-z|)}\wedge 1\).$$
If $\delta_D(y) < \delta_D(z)/2$, then $r(y,z) = \delta_D(z) < 2|y-z|$. Hence, by \eqref{subadd},
\begin{eqnarray*}
\frac{V(\delta_D(y))V(\delta_D(z))}{V^2(|y-z|)}\wedge 1 \leq  \frac{V(\delta_D(y))V(\delta_D(z))}{V^2\(r(y,z)/2\)}
 \leq 4\, \frac{V(\delta_D(y))V(\delta_D(z))}{V^2(r(y,z))}, \quad y,z\in D.
\end{eqnarray*}
\end{proof}

The following result is the so-called $3G$-theorem (see \cite{MR2892584}).
\begin{proposition}\label{lem:3G}Let $D$ be a bounded open $C^{1,1}$ at scale $r>0$.
There is a constant $\CII=\CII(d,\psi,\diam(D)/r)$ such that
$$\frac{ G_D(x,z) G_D(z,y)}{ G_D(x,y)}\leq \CII V(\delta_D(z)) \(\frac{ G_D(x,z)}{V(\delta_D(x))}\vee \frac{ G_D(z,y)}{V(\delta_D(y))}\).$$
\end{proposition}
\begin{proof}
Let $\mathcal{G}(x,y)=U(x,y)/V^2(r(x,y))$. Then,
\begin{equation}\label{GreenCal3G}\mathcal{G}(x,z)\wedge\mathcal{G}(z,y)\leq c(d) \mathcal{G}(x,y).\end{equation}
Indeed, assume that $|y-z|\leq |x-z|$, then $|x-y|\leq 2|x-z|$ and
$$r(x,y)\leq \delta_D(x)+|x-y|\leq 3r(x,z).$$
By monotonicity of $U$, $V$ and \eqref{subadd} we obtain
$$\mathcal{G}(x,z)\leq \frac{U((x-y)/2)}{V^2(r(x,y)/3)}\leq 3^22^d\mathcal{G}(x,y).$$

By Lemma \ref{GreenEst1}, $\mathcal{G}_D(x,y) \approx G_D(x,y) / (V(\delta_D(x)V(\delta_D(y))$. Hence, by \eqref{GreenCal3G},
\begin{align*}
\frac{ G_D(x,z) G_D(z,y)}{ G_D(x,y)}&\approx  V^2(\delta(z))\frac{\mathcal{G}(x,z)\mathcal{G}(z,y)}{\mathcal{G}(x,y)}\leq c V^2(\delta(z)) \(\mathcal{G}(x,z)\vee\mathcal{G}(z,y)\)\\
&\approx V(\delta_D(z))\(\frac{ G_D(x,z)}{V(\delta_D(x))}\vee \frac{ G_D(z,y)}{V(\delta_D(y))}\).
\end{align*}
\end{proof}

The next lemma is crucial in our consideration. The proof is based on the proof of \cite[Lemma 9]{MR2892584}. Nevertheless, we give the details, because here we can see, how the weak scaling condition is used.

\begin{lemma}\label{GDUnif}Let $0 < r_0 < \infty$ and $\diam D\leq r_0$. Then
$G_D(y, z)/[\delta_D(z) \wedge |y-z|]$ is uniformly in $y$ integrable against $|b(z)|dz$.
\end{lemma}
\begin{proof}
By Lemma \ref{GreenEst1}, it is enough to prove the uniform integrability of  $$H(y,z)=U(y-z)\frac{V(\delta_D(y))V(\delta_D(z))}{V^2(r(y,z))}\frac{\delta_D(z)\vee |y-z|}{|y-z|\delta_D(z)}.$$
Let $A_R(y)=\{z \in D \colon H(y,z)>R\}$. We will show that
$$
\lim_{R\to\infty}\sup_{y \in D} \int_{A_R(y)}H(y,z)|b(z)|dz=0.
$$
Let $\lC_2=\lC_2(\diam(D))$ be such that
\begin{equation}
V(\eta r) \le \lC_2 \eta^{\la_1/2} V(r), \qquad \eta<1, \; r< \diam(D)
\label{eq:scalV2v2}
\end{equation}
(see Remark \ref{remScalExp}). We recall that $\la_1>1$. For $r>0$, we denote
\begin{equation}\label{eq:supPot}
K_r=\sup_{x\in \Rd} \int_{B(x,r)}|b(y)|\frac{U(x-y)}{|x-y|}dy\,.
\end{equation}
By (\ref{eq:Kc}), 	$K_r<\infty$ and $K_r \downarrow 0$ as $r \downarrow 0$.
Since $U$ is radial decreasing function, we may denote $U(r) = U(x)$ for all $|x|=r$ and we have
$$
\int_{B(x,r)} |b(z)|\,dz \le \frac{r}{U(r)} \int_{B(x,r)}
\frac{U(x-z)}{|x-z|}|b(z)|\,dz \le K_r \frac{r}{U(r)}\,, \quad x \in \Rd, \;r>0\,.
$$
Let $m \ge 2$ be such that $\delta_D(y)\leq m\delta_D(z)$, then by (\ref{eq:scalV2v2}),
\begin{align}
H(y,z)\frac{|y-z|}{U(y-z)}&\leq \frac{V\(\frac{\delta_D(y)}{r(y,z)}r(y,z)\)V\(\frac{\delta_D(z)}{r(y,z)}r(y,z)\)}{V\(r(y,z)\)^2}\frac{r(y,z)}{\delta_D(z)} \notag\\
\leq\lC_2^2\frac{\delta_D(y)^{\la_1/2}}{r(y,z)^{\la_1-1}\delta_D(z)^{1-\la_1/2}}&\leq\lC_2^2\(\frac{\delta_D(y)}{\delta_D(z)}\)^{1-\la_1/2}\leq  \lC_2^2m^{1-\la_1/2}. \label{eq:Hest}
\end{align}
By (\ref{eq:scalV2v2}), we also have
\begin{equation}
\frac{U(y-z)}{|y-z|} = \frac{V^2(y-z)}{|y-z|^{d+1}} \le \lC_2^2 \frac{|y-z|^{\la_1}}{\diam(D)^{\la_1}} V^2(\diam(D)).
\label{eq:Uest}
\end{equation}
Hence, (\ref{eq:Hest}) yields $A_R(y) \subset \{z \in D \colon |y-z| < cR^{-1/(d+1-\la_1)}\}$, where $c = c(m, (\diam(D), \la_1)$ is some constant.

Let $D_{r}=\{x\in D:\delta_D(x)\geq r\}$. If $R \to \infty$, then uniformly in $y$,
\begin{equation}
\int_{A_R(y)\cap D_{\delta_D(y)/m}}H(y,z)|b(z)|dz\leq \lC_2^{2}m^{1-\la_1/2}K_{c R^{-1/(d+1-\la_1)}}\to 0\label{eq:unif1}.
\end{equation}

For $y\in D$, \ $k,n \ge 0$ and $m \ge 2$, we consider
$$
W^m_{n,k}(y) = \left\{z \in D\colon
\frac{\delta_D(y)}{m2^{n+1}} < \delta_D(z) \le \frac{\delta_D(y)}{m2^{n}}
,\; k < \frac{|y-z|}{\delta_D(y)}\le (k+1)  \right\}.
$$
$W^m_{n,k}(y)$ may be covered by $c_1 (k + 1)^{d-2} m^{d-1}2^{n(d-1)}$ balls of
radii $\frac{\delta_D(y)}{m2^n}$, thus
\begin{align*}
  \int_{W^m_{n,k}(y)}&|b(z)|\,dz \\
  & \le c_1 (k + 1)^{d-2} m^{d-1}2^{n(d-1)} \sup_{x \in \RR^d} \int_{B(x,\delta_D(y)/m2^n)} |b(z)|\,dz \\
  & \le c_1 K_{\delta_D(y)/m2^n} (k + 1)^{d-2} m^{d-1}2^{n(d-1)} \(\frac{\delta_D(y)}{m2^n}\)^{d+1}V^{-2}\(\frac{\delta_D(y)}{m2^n}\)\\
  & = c_1 K_{\delta_D(y)/m2^n} (k + 1)^{d-2} m^{-2}
  2^{-2n}\delta_D(y)^{d+1}V^{-2}\(\frac{\delta_D(y)}{m2^n}\).
\end{align*}
For $z\in W^m_{n,k}(y)$, we have $\delta_D(y)\geq 2\delta_D(z)$, hence $|y-z|\ge \delta_D(y)/2$ and $|y-z|\ge\delta_D(z)$. Therefore,
$$
H(y,z)\leq \frac{V(\delta_D(y))V(\delta_D(z))}{|y-z|^d\delta_D(z)}, \qquad z \in W^m_{n,k}(y),
$$
and we obtain
\begin{align*}
  &\int\limits_{A_R(y) \setminus D_{\delta_D(y)/m}} H(y,z) |b(z)|\,dz
\le \sum_{n=0}^\infty\sum_{k=0}^{\infty}\int\limits_{W^m_{n,k}(y)}\frac{V(\delta_D(y))V(\delta_D(z))}{|y-z|^d\delta_D(z)}|b(z)|\,dz\\
  &\le  \sum_{n=0}^\infty\sum_{k=0}^{\infty}\frac{V(\delta_D(y))V\(\delta_D(y)/(m2^n)\)m2^{n+1}}{\((k+1)\delta_D(y)/2\)^d\delta_D(y)}\int_{W^m_{n,k}(y)} |b(z)|\,dz\\
  & \le c_2K_{\delta_D(y)/m}\sum_{n=0}^\infty\sum_{k=0}^{\infty}(k + 1)^{-2} m^{-1}2^{-n}
  \frac{V(\delta_D(y))}{V\(\frac{\delta_D(y)}{m2^n}\)}\\&\le c_3 K_{\delta_D(y)/m}\sum_{n=0}^\infty\sum_{k=0}^{\infty}(k + 1)^{-2} m^{\ua/2-1}2^{n(\ua/2-1)}\le
  c_4m^{\ua/2-1}K_{\delta_D(y)/m}\,.
\end{align*}
Let $\varepsilon >0$. We chose $m$  and $R$ so large that $c_4m^{\alpha/2-1}K_{\diam(D)/m} <
\varepsilon/2 $ and
$$
\sup_{y \in D}\int_{D_{\delta_D(y)/m} \cap A_R(y)} H(y,z)
|b(z)|\,dz <\varepsilon/2\,.
$$
This completes the proof.
\end{proof}

\begin{lemma}\label{GradGreen}
  If $f \in \pK_d$, then
  \begin{equation}
    \nabla_y\int_D  G_D(y,z)f(z)\,dz = \int_D \nabla_y\, G_D(y,z)f(z)\,dz\,,\quad y \in D\,.
  \end{equation}
\end{lemma}
\begin{proof}
Fix $y \in D$ and let $0<h<\delta_D(y)/2$ and $h_d=(0,\ldots,0,h)\in\Rd$.
Then,
\begin{align*}
\frac{| G_D(y+h_d,z)- G_D(y,z)|}{h} &= \frac{1}{h}\left|\int_0^1\frac{d}{ds} G_D(y+s h_d,z) ds \right| \\
 =  \left|\int_0^1\frac{\partial}{\partial y_d} G_D(y+s h_d,z) ds \right|  &\le  c_1 \int_0^1\frac{G_D(y+s h_d,z)}{|y+s h_d-z| \land \delta_D(y+sh_d)} ds \\
& \leq  c_2 \int_0^1\frac{U(y+s h_d,z)}{|y+s h_d-z|} ds\,.
\end{align*}
Since $f\in\pK_d$,  $\frac{U(y+s h_d,z)}{|y+s h_d-z|}$ is uniformly in $h$ integrable on $(0,1) \times D$, which ends the proof (see \cite[Lemma 10]{MR2892584}).
\end{proof}

For $x, y \in D$, we let
\begin{equation}\label{def:kappa}
\kappa(x,y) = \int\limits_D|b(z)|\frac{ G_D(x,z) G_D(z,y)}{ G_D(x,y)(\delta_D(z)\wedge|y-z|)}dz,
\end{equation}
\begin{equation}\label{def:wKappa}
\widehat{\kappa}(x,y) = \int\limits_D|b(z)|\frac{ G_D(x,z) G_D(z,y)(\delta_D(x)\wedge|x-y|)}{ G_D(x,y)(\delta_D(z)\wedge|y-z|)(\delta_D(x)\wedge|x-z|)}dz.
\end{equation}

\begin{lemma}\label{lem:boundKappa}
Let $\lambda <\infty, r < 1$. There is
$\CIII = \CIII(d,\tnu, b,\lambda,r)$
such that if D is $C^{1,1}$,
$\diam(D)/r_0(D) \leq \lambda$ and $\diam(D)\leq r$, then $\kappa(x,y) \leq \CIII$, $\widehat{\kappa}(x,y)\leq 2\CIII$ for $x,y \in D$, and $\CIII(d,\nu,b,\lambda,r) \rightarrow 0$ as $r\rightarrow 0$.
\end{lemma}
\begin{proof}

By Lemma \ref{GreenEst1} and (\ref{scalV2}), we have 
\begin{align*}
\frac{V(\delta_D(z))}{V(\delta_D(x))} G_D(x,z) & \approx  \frac{V^2(\delta_D(z))}{V^2(r(x,z)))}U(x-z) \\
&\leq c\left(\frac{\delta_D(z)}{r(x,z)}\right)^{\la_1} U(x-z) \leq C(\delta_D(z) \wedge |x-z|)\frac{U(x-z)}{|x-z|}.
\end{align*}
By Proposition \ref{lem:3G}, we obtain
\begin{multline}\label{eq:ineqGGGUU}
\frac{ G_D(x,z) G_D(z,y)}{ G_D(x,y)} \leq \CII V(\delta_D(z))\left(\frac{ G_D(x,z)}{V(\delta_D(x))} \vee \frac{ G_D(z,y)}{V(\delta_D(y))}\right) \\
\leq c\CII \left((\delta_D(z) \wedge |x-z|)\frac{U(x-z)}{|x-z|}\right) \vee \left((\delta_D(z) \wedge |y-z|)\frac{U(y-z)}{|y-z|}\right) \\
= C\CII (\delta_D(z) \wedge |x-z| \wedge |y-z|) \left( \frac{U(x-z)}{|x-z|} \vee \frac{U(y-z)}{|y-z|} \right). 
\end{multline}
Hence,
$$
\frac{ G_D(x,z) G_D(z,y)}{ G_D(x,y)(\delta_D(z)\wedge |y-z|\wedge |x-z|)}\leq C\CII  \left(\frac{U(x-z)}{|x-z|} \vee \frac{U(y-z)}{|y-z|} \right).
$$
By (\ref{eq:supPot}) and observation that $\lim\limits_{r\rightarrow 0}K_r = 0$, we have the statement for $\kappa$.
The rest of the proof is the same as \cite[Lemma 11]{MR2892584}, so we omit it.
\end{proof}


\subsection{Green function of $\tilde{\mathcal{L}}$}
We will consider analogous objects to the ones considered in the previous section.
We define the Green function and the Green operator
of $\tilde{\mathcal{L}} = \mathcal{L}+b\nabla$ on $D$
\begin{align}\label{eq:12.5t}
&\tG_D(x,y)=\int_0^\infty \tp_D(t,x,y)dt \,, \qquad x,y \in \RR^d\,,\\
&\tG_D \phi(x)=\int_\Rd \tG_D(x,y)\phi(y)dy\,, \qquad \phi \in C^c(\RR^d)\,. \notag
\end{align}
From the properties of  $\tp_D(t,x,y)$ we get that $\tG_D(x,y)=0$ if $x\in D^c$ or $y\in D^c$.

By (\ref{ptxy_comp}), we have
$$
\lim_{t \to 0} \frac{\tp(t,x,y)}{t} = \lim_{t \to 0}
\frac{p(t,x,y)}{t} = \nu(y-x)\,.
$$
Thus, the intensity of jumps of the canonical process $X_t$ is the same as $\tX_t$. Accordingly, we obtain the following description.

\begin{lemma}\label{lem:lsgp}
The $\tPP^x$-distribution of $(\tau_D,X_{\tau_D})$ on $(0,\infty)\times (\overline{D})^c$ has density
\begin{equation}
  \label{eq:IWft}
  \int_D \tp_D(u,x,y)\nu(z-y)\,dy\,,\quad u>0\,,\;z \in (\overline{D})^c\,.
\end{equation}
\end{lemma}

\noindent We define the Poisson kernel of $D$ for $\tilde{\mathcal{L}}$,
\begin{equation}\label{eq:djpt}
\tilde{P}_D(x,y)=\int_D \tG_D(x,z)\nu(y-z)\,dz\,,\quad x\in D\,,\;y\in D^c\,.
\end{equation}
By (\ref{eq:12.5t}), (\ref{eq:djpt}) and (\ref{eq:IWft}),  we have
\begin{equation}\label{eq:IWt}
\tPP^x(X_{\tau_D} \in A) 	 =\int_A \tP_D(x,y)dy\,,
\end{equation}
if $A \subset (\bar{D})^c$. For the case of $A\subset \partial D$, we refer the reader to Lemma~\ref{l:nu}.


\begin{lemma}\label{lem:GLDupper}
$\tG_D(x,y)$ is continuous for $x\neq y$, $\tG_D(x,x) = \infty$ for $x\in D$, and
$$
\tG_D(x,y)\leq \CXXI U(x-y),\textsf{~~~ $x, y \in \mathbb{R}^d$},
$$
where $\CXXI = \CXXI (d,b,\diam(D))$.
\end{lemma}
\noindent Since the proof is the same as the proof of \cite[Lemma 7]{MR2892584}, we omit it.

For $x\neq y$, we let
$$ G_1(x,y) = \int_DG_D(x,z)b(z)\cdot\nabla_z G_D(z,y)dz.$$
By Lemma \ref{lem:boundKappa},
\begin{equation}\label{kxyEstimates}
| G_1(x,y)|\leq \Cg G_D(x,y)\int_D\frac{|b(z)| G_D(x,z) G_D(z,y)}{ G_D(x,y)(\delta_D(z)\wedge|y-z|)}dz \leq
\Cg \CIII G(x,y).
\end{equation}
For $f\in\pK_d$, we have
\begin{align*}
&\int_D G_D(x,y)\int_D|b(z)|\frac{ G_D(x,z) G_D(z,y)}{ G_D(x,y)(\delta_D(z)\wedge|y-z|)}dz|f(y)|dy \\
& \leq \CIII\int_D G_D(x,y)|f(y)|dy < \infty.
\end{align*}
Hence, by Lemma \ref{GradGreen}, (\ref{eq:estGradGreena}) and Fubini's theorem,
$$
 G_Db\nabla  G_D f (x) = \int_D G_D(x,z)\int_Db(z)\cdot\nabla G_D(z,y)f(y)dydz \\ = \int_D G_1(x,y)f(y)dy.
$$
We like to note that linear map $f \mapsto b\nabla G_D f$ preserves $\pK_d$ because $\nabla G_D f$ is a bounded function, see Lemma \ref{GDUnif} for $b$ equals $f$. 

The next lemma results from integrating (\ref{tp_D_fs}) against time.
\begin{lemma}\label{lem:tgi}
For all $\varphi \in C_c^\infty(D)$ and $x\in D$, we have
\begin{equation}\label{tG_D_fs}
\int_{D} \tG_D(x,z) \tL \varphi(z) \,dz = \int_{D} \tG_D(x,z) \left(\mathcal{L}\varphi(z) + b(z)\cdot\nabla\varphi(z)\right) \,dz = - \varphi(x)\,.
\end{equation}
\end{lemma}

\noindent For every $x\in D$, let us define the function
\begin{equation}\label{l:pf}
f_x(y) = \tG_D(x,y)-G_D(x,y)-\int_D\tG(x,z)b(z)\cdot\nabla_zG_D(z,y)dz.
\end{equation}
We can notice that $f_x(y) = 0$ for $y\in \overline{D}^c$.

\begin{lemma}
$f_x(y)$ is well defined on $\mathbb{R}^d\setminus\{x\}$, integrable on $\mathbb{R}^d$ and bounded on $\mathbb{R}^d\setminus B(x,r)$ for $0<r$.
\end{lemma}

\begin{proof}
Let us fix $y\neq x$ 
and $0 < \rho \leq \min\{\frac{|x-y|}{2},\frac{\delta_D(x)}{2}\}$.  By Lemma \ref{lem:GLDupper} and (\ref{eq:estGradGreena})
\begin{equation}\label{Eq:fxWD}
\int_D|\tG_D(x,z)b(z)\cdot\nabla_zG_D(z,y)|dz \leq C_4\Cg\int_D U(x-z)|b(z)|\frac{|G_D(z,y)|}{\delta_D(z)\wedge|z-y|}dz.
\end{equation}
Let $D = D_1 \cup D_2$, where $D_1 = B(x,\rho/2)^c\cap D$ and $D_2 =  B(x,\rho/2)$. By monotonicity of $U$ and  Lemma \ref{GDUnif},
\begin{multline}\label{Eq:Grho}
\int_{D_1} U(x-z)|b(z)|\frac{|G_D(z,y)|}{\delta_D(z)\wedge|z-y|}dz\\ \leq U\left(\frac{\rho}{2}\right)\int_{D}|b(z)|\frac{|G_D(z,y)|}{\delta_D(z)\wedge|z-y|}dz \leq c_1 U\left(\frac{\rho}{2}\right),
\end{multline}
for every $y\in D$. Since $b\in\pK_d$,
\begin{equation}\label{Eq:UGrho}
\int_{D_2} U(x-z)|b(z)|\frac{G_D(z,y)}{\delta_D(z)\wedge|z-y|}dz \leq \frac{\CXXI U(\rho)}{\rho}\int_{D}U(x-z)|b(z)|dz \leq c_2\frac{U(\rho)}{\rho}.
\end{equation}
It implies that (\ref{Eq:fxWD}) is finite for every $y\neq x$ and bounded on $\mathbb{R}^d\setminus B(x,r)$ for every $r>0$.

It remains to show the integrability of $f_x$.
Let $r=\frac{\delta_D(x)}{4}$ and $B=B(x,2r)$. We put $M_r = (2c_1U(\frac{r}{2})+c_2\frac{U(r)}{r})|D|$. By (\ref{Eq:Grho}) and (\ref{Eq:UGrho}),
%
%
\begin{align*}
&\int_D \int_D|\tG_D(x,z)b(z)\cdot\nabla_zG_D(z,y)|dz dy\\
&\leq M_r + \int_B \int_B|\tG_D(x,z)b(z)\cdot\nabla_zG_D(z,y)|dz dy \\
&\leq M_r +  \Cg^2\int_B\int_{B}U(x-z)|b(z)|\frac{G_D(z,y)}{\delta_D(z)\wedge|z-y|}dzdy\\
& \leq  M_r + \Cg^2\int_B\int_B U(x-z)|b(z)|\frac{U(z-y)}{|z-y|}dzdy \\
&\leq M_r +  \Cg^2\int_{D}U(x-z)|b(z)|\int_{B(z,\diam(D))}\frac{U(z-y)}{|z-y|}dydz \\ 
& \leq M_r +  c_D\int_{D}U(x-z)|b(z)|dz\,,	
\end{align*}

which is finite since $b\in\pK_d$.
\end{proof}

\begin{theorem}\label{lem:pf}
Let $x,y\in\mathbb{R}^d$, $x\neq y$. We have
\begin{equation}\label{eq:wp}
\tG_D(x,y) = G_D(x,y)+\int_D\tG_D(x,z)b(z)\cdot\nabla_zG_D(z,y)dz.
\end{equation}
\end{theorem}
\begin{proof}
For $x \not \in D$, $G_D(x,\cdot) \equiv 0$ and (\ref{eq:wp}) follows. We fix $x \in D$.
Let $g\in C_c^\infty(B(0,1))$ be a symmetric function such that  $g\geq 0$ and $\int g(x)dx=1$. Let $r_x = \frac{\delta_D(x)}{3} > \delta > 0$ and $g_\delta(x) = \delta^{-d}g(\frac{x}{\delta})$. Set
$$
D_{+\delta}=\{x:\dist(x,{D})<\delta\} \quad \mbox{and} \quad D_{-\delta}=\{x\in D:\dist(x,\partial{D})>\delta\}.
$$
We consider $u_{\delta, x}=g_\delta*f_x\in C_c^\infty(D_{+\delta})$.
Let $\varphi\in C_c^\infty(D_{-\delta})$, then  $g_\delta*\varphi\in C_c^\infty(D)$. By Lemma \ref{lem:tgi},
\begin{equation}
\langle g_\delta*f_x,\mathcal{L}\varphi\rangle=\langle f_x,g_\delta*\mathcal{L}\varphi\rangle = \langle f_x,\mathcal{L}(g_\delta*\varphi)\rangle = 0.
\end{equation}
So $u_{\delta,x}$ is weak $\mathcal{L}$-harmonic on $D_{-\delta}$. Since $u_{\delta,x}\in\mathcal{D}(\mathcal{L})$ by \cite[Theorem 2.7]{MR2515419}
$
u_{\delta,x}(y) = \mathbb{E}^yu_{\delta,x}(X_{\tau_U})$ for every $ \overline{U}\subset D_{-\delta}$.
Since $\delta < r_x$, for every $y \in \RR^d$ , we have
$$
|u_{\delta,x}(y)|\leq\mathbb{E}^y|u_{\delta,x}(X_{\tau_{B(x,2r_x)}})| \leq \Vert f_x\mathbbm{1}_{B^c(x,r_x)}\Vert_\infty:=M.
$$
Since $|u_{\delta,x}(y)|\xrightarrow[\delta \to 0]{} |f_x(y)|$ a.s., we obtain $|f_x(y)| \leq \Vert f_x\mathbbm{1}_{B^c(x,r_x)}\Vert_\infty$, a.s. Since $f_x$ is continuous,  $f_x$ is bounded on $\mathbb{R}^d$. \\
Let $\{U_n\}_{n\in\mathbb{N}}$ be a family of sets such that $U_n \nearrow D_{-\delta}$. By quasi-left continuity of $X_t$,
\begin{align*}
 |u_{\delta,x}(y)| &=|\lim_{n\to\infty}\mathbb{E}^yu_{\delta,x}(X_{\tau_{U_n}})|= |\mathbb{E}^y\lim_{n\to\infty}u_{\delta,x}(X_{\tau_{U_n}})| = |\mathbb{E}^yu_{\delta,x}(X_{\tau_{D_{-\delta}}})|\\
&= |\mathbb{E}^y(u_{\delta,x}(X_{\tau_{D_{-\delta}}}),X_{\tau_{D_{-\delta}}} \in D_{+\delta}\setminus D_{-\delta})|\leq M\mathbb{P}^y(X_{\tau_{D_{-\delta}}}\in D_{+\delta}\setminus D_{-\delta}).
\end{align*}
So $|u_{\delta,x}(y)|\leq M\mathbb{P}^y(X_{\tau_{D_{-\delta}}}\in D_{+\delta}\setminus D_{-\delta})$ and with $\delta \rightarrow 0$, we finally obtain
\begin{equation*}
|f_x(y)|\leq M\mathbb{P}^y(X_{\tau_D}\in\partial D)=0,
\end{equation*}
which  completes the proof.
\end{proof}

Let $\kk_0(x,y)=G_D(x,y)$. We inductively define
\begin{equation*}
\kk_{n}(x,y)=\int_D \kk_{n-1}(x,z) b(z)\cdot \nabla_z G_D(z,y)\,dz\,,\quad x\neq y\in D\,,\quad n=1,2,\ldots\,.
\end{equation*}
By Lemma~\ref{GradGreen}, Lemma~\ref{lem:boundKappa}, Fubini's theorem and induction, we also have 
\begin{equation}\label{eq:gn2}
G_n(x,y)=\int G_D(x,z)b(z)\cdot \nabla_z G_{n-1}(z,y)dz\,, \quad x\neq y\in D, n=2,3,\ldots\,.
\end{equation}
 We end this section with the estimates of $\tG_D(x,y)$ for small sets $D$.
\begin{lemma}\label{Theorem1s} 
Let $d\geq 2$, $b\in \pK_d$ and $\lambda>0$. There is $\varepsilon=\varepsilon(d,\nu,b,\lambda)>0$ such that if
${\rm diam}(D)/r_0(D)\le \lambda$ and $\diam(D)\le\varepsilon$, then
  \begin{equation}
    \label{eq:egfs}
\frac{2}{3}G_D(x,y) \le \tG_D(x,y) \le \frac{4}{3} G_D(x,y), \qquad x,y \in
\Rd\,.
  \end{equation}
\end{lemma}
\begin{proof}
We follow the arguments from \cite{MR2892584}. We present only the main steps of the proof, the details are left to the reader.

Let $x \not= y$.
Iterating (\ref{eq:wp}), by (\ref{eq:gn2}), we obtain for $n=0, 1,\ldots$,
\begin{eqnarray}
\tG_D(x,y)&=&G_D(x,y)+\int \tG_D(x,z)b(z)\cdot \nabla_zG_D(z,y)dz\nonumber\\
&=&\sum_{k=0}^n \kk_k(x,y)
+\int \tG_D(x,z)b(z)\cdot \nabla_{z}G_n(z,y)dz\,.\label{eq:ipf}
\end{eqnarray}

Let $\lambda>0$. We note that the constant $\CIII$ from Lemma \ref{lem:boundKappa} may be arbitrary small if $\diam(D)/r_0(D)\le \lambda$ and $r_0(D)$ is small enough. Hence, we may choose $\varepsilon=\varepsilon(d,\nu,b,\lambda)>0$ such that $\Cg\CIII < 1/4$. By (\ref{gradEst}),  (\ref{kxyEstimates}),  Lemma \ref{lem:boundKappa} and induction,
\begin{align}
  |\kk_{n}(x,y)|  &\le  \int_D  |\kk_{n-1}(x,z)| |b(z)| |\nabla_z G_D(z,y)| \,dz \nonumber  \\
   &\le  (\Cg\CIII)^{n-1} \int_D    G_D(x,z) |b(z)|  |\nabla_z G_D(z,y)| \,dz  \le 4^{-n} G_D(x,y)\,,  \label{eq:eGn} \\
&|\nabla_x G_{n}(x,y)|  \leq 2^{-n}\Cg \frac{G_D(x,y)}{\delta_D(x) \land |x-y|}\,, \label{eq:gn1}
\end{align}
for $n=0,1, 2, \ldots\,.$ Now, we have $\tG_D(x,y)=\sum_{n=0}^\infty \kk_n(x,y)$. 
Indeed, by (\ref{eq:gn1}), the remainder in (\ref{eq:ipf}) is bounded by
\begin{eqnarray*}
 2^{-n}\Cg \int_D U(x-z)|b(z)|  \frac{G_D(z,y)}{\delta_D(z)\land |y-z|}dz\to 0\,,\quad \mbox{ as } n\to \infty\,.
\end{eqnarray*}
The integral is finite because of Lemma~\ref{GDUnif}. Thus, by (\ref{eq:eGn}),
\begin{align*}
\tG_D(x,y)
&\leq\sum_{n=0}^\infty \kk_n(x,y)\le
\sum_{n=0}^\infty 4^{-n}G_D(x,y) = \frac{4}{3}G_D(x,y)\,,\\
\tG_D(x,y)&\geq
G_D(x,y) - \sum_{n=1}^\infty 4^{-n}G_D(x,y)=\frac{2}{3}G_D(x,y)\,.
\end{align*}
\end{proof}

\section{Proof of Theorem~\ref{Theorem1}}\label{sec:b}

Using the comparability of $G_D$ and $\tG_D$ for small $C^{1,1}$ sets and repeating the arguments from \cite{MR2892584}, we obtain estimates of the Poisson kernel and Harnack principles. The proofs are almost identical to the ones from \cite{MR2892584}. Nevertheless, due to the references we use, we present them below.

By Ikeda-Watanabe formula, we get
\begin{equation}
  \tPP^x(X_{\tau_D} \in A) \approx \PP^x(X_{\tau_D} \in A)\,,\quad x\in D\,,\quad A\subset (\overline{D})^c\,. \label{I-WComp1}
\end{equation}
for sufficiently small $\diam(D)$ and bounded distortion. The next lemma says that the process $\tX_t$ does not hit the boundary of our general $C^{1,1}$ open set $D$ in the moment of the first exit from $D$.

\begin{lemma}\label{l:nu} 
For every $x\in D$, we have $\tPP^x(X_{\tau_D}\in \partial D)=0$.
\end{lemma}
\begin{proof}
Let $u(x)=\tPP^x(X_{\tau_D}\in \partial D)$, $x\in \Rd$.
We claim that there exists $c=c(d,\nu,D, b)>0$ such that $u(x)<1-c$ for $x\in D$.
Indeed, we consider small $\varepsilon>0$, $x\in D$,  $r=\varepsilon\, {\rm dist}(x,D^c)$, the ball $B=B(x,r/2)\subset D$, and a ball $B'\subset (\overline{D})^c$ with radius and distance to $B$ comparable with $r$. By (\ref{I-WComp1}), (\ref{ptxy_comp}) and Lemma \ref{lem:lsgp}
$$\tPP^x(X_{\tau_D}\notin \partial D)
\geq \tPP^x(X_{\tau_B(x,r/2)}\in B')\approx \PP^x(X_{\tau_B(x,r/2)}\in B')
\geq c\,,$$
where in the last inequality we used \eqref{eq:Pk}, \eqref{nu_L}, \eqref{subadd} and \cite{MR632968}.
 Furthermore, let $D_n=\{y\in D:\,{\rm dist}(y,D^c)>1/n\}$, $n=1,2,\ldots$.
We consider $n$ such that $B(x,r/2)\subset D_n$.
We have
$\tPP^x(X_{\tau_{D_n}}\in \overline{D})\leq 1- \tPP^x(X_{\tau_B}\in B')\leq 1-c$, as before.
Let $C=\sup \{u(y): y\in D\}$. We have
$u(x)=\tEE^x\{u(X_{\tau_{D_n}});\,X_{\tau_{D_n}}\in \overline{D}\}\leq C(1-c)$, hence
$C\leq C(1-c)$ and so $C=0$.
\end{proof}
In the context of Lemma~\ref{Theorem1s},
the $\tPP^x$ distribution of $X_{\tau_D}$ is absolutely continuous with respect to the Lebesgue measure,
and has density function
\begin{equation}\label{PoissonComp}
  \tilde P_D(x,y) \approx P_D(x,y)\,,\quad \;y\in D^c\,,
\end{equation}
provided $x\in D$.
This follows from (\ref{eq:IWt}) and Lemma~\ref{l:nu}.
For clarity,
\begin{equation}
  \tPP^x(X_{\tau_S} \in A) \approx \PP^x(X_{\tau_S} \in A)\,,\quad x\in S\,,\quad A\subset S^c\,. \label{I-WComp}
\end{equation}

\begin{lemma}[Harnack inequality for $\tL$]\label{HIforL} 
  Let $x,y \in \Rd$, $0<s<1$ and $k\in \mathbb{N}$ satisfy $|x-y|
  \le 2^ks$. Let ${u}$ be non-negative in $\Rd$ and
  $\tilde{\mathcal{L}}$-harmonic in $B(x,s) \cup B(y,s)$.  There is $\CIV = \CIV(d, \psi, b)$ such that
  \begin{equation}\label{LHarnackIneq}
    \CIV^{-1}2^{-k(d+\ua)}{u}(x) \le {u}(y) \le \CIV 2^{k(d+\ua)}{u}(x)\,.
  \end{equation}
\end{lemma}
\begin{proof}

We may assume that $s\leq 1\land \varepsilon/2$, with  $\varepsilon$ of Lemma~\ref{Theorem1s}.
Let $f(z)=u(z)$ for $z\in B(y,2s/3)^c$ and $f(z)=\int_{B(y,2s/3)^c} {u}(v)P_{B(y,2s/3)}(z,v)\,dv$ for $z\in B(y,2s/3)$, so that
$f$ is non-negative in $\RR^d$ and $\mathcal{L}$-harmonic in $B(y,2s/3)$. Let $z \in  B(y,s/2)$.
By (\ref{I-WComp}),
 $$
    {u}(z) = \tEE^z {u}(X (\tau_{B(y,2s/3)})) = \int_{B(y,2s/3)^c} {u}(v)\tilde P_{B(y,2s/3)}(z,v)\,dv\approx f(z)\,.
 $$
The Harnack inequality for $\mathcal{L}$ (\cite{MR3225805}) implies $u(y)\approx u(z)$, where the comparability constant depends on $\psi$, $d$ and $b$. The standard chain rule provides  $u(x)\approx u(y)$ for $|x-y|< 3/2 s$.  Therefore we assume that $|x-y| \ge 3s/2$.  For $z \in B(y,s/2)$  and $w \in B(x,s/2)$ we have $|w-z| \le   |x-y|+|y-z|+|w-x|
 \le2^k s + s  \le 2^{k+1}s$. Hence by the Ikeda-Watanabe formula, \eqref{nu_L} and \cite{MR632968}
  \begin{align*}
	P_{B(x,s/2)}(x,z) &=\int_{B(x,s/2)}G_{B(x,s/2)}(x,w)\nu(|w-z|)dw \geq \mathbb{E}^x\tau_{B(x,s/2)}\nu(2^{k+1}s) \\
	 & \approx \frac{\psi\left(\frac{1}{2^{k+1}s}\right)}{(2^{k+1}s)^d\psi\left(\frac{2}{s}\right)} \geq \frac{1}{(2^{k+1}s)^d\uC 2^{\ua(k+2)}}  \geq \frac{1}{2^{k(d+\ua)}\uC 2^{d+2\ua}}s^{-d}.
  \end{align*}
Since
   $\tilde P_{B(x,s/2)} \approx P_{B(x,s/2)}$, by the first part of the proof we obtain
  \begin{multline}
    {u}(x)  =  \int_{B(x,s/2)^c} \tilde P_{B(x,s/2)}(x,z){u}(z)\,dz
\ge \int_{B(y,s/2)} \tP_{B(x,s/2)}(x,z) {u}(z) \,dz \\\notag
    \approx \int_{B(y,s/2)} P_{B(x,s/2)}(x,z) {u}(z) \,dz \ge  \frac{c|B(y,s/2)|}{2^{k(d+\ua)}s^d 2^{d+2\ua}} {u}(y)
    =  \CIV 2^{-k(d+\ua)}{u}(y)\,.
  \end{multline}
  By symmetry, $u (x)\approx u(y)$.
\end{proof}
We obtain also the boundary Harnack principle for $\mathcal{L}$ and general $C^{1,1}$ sets $D$.

\begin{lemma}[BHP]\label{BHPforL} 
  Let $z \in \partial{D}$, $0<r\le  r_0(D)$, and $0<p<1$. If
$\tilde{u}, \tilde{v}$ are non-negative in $\Rd$, regular $\tilde{\mathcal{L}}$-harmonic
  in $D \cap B(z,r)$, vanish on $D^c \cap B(z,r)$
  and satisfy $\tilde{u}(x_0)=\tilde{v}(x_0)$ for some $x_0 \in D \cap B(z,pr)$
  then
  \begin{equation}
    \label{BHPforLEq}
    \CVI^{-1}\tilde{v}(x) \le \tilde{u}(x) \le \CVI \tilde{v}(x)\,,\quad x \in D \cap B(z,pr)\,,
  \end{equation}
with $\CVI = \CVI(d,\psi,b,p,r_0(D))$.
\end{lemma}
\begin{proof}  In view of  Lemma~\ref{HIforL} we may assume that $r$ is small.
Let $F=F(z,r/2) \subset B(z,r)$ be the $C^{1,1}$ domain of Lemma~\ref{l:loc}, localizing $D$ at $z$.
For $x
  \in F$ we have $\tilde{u}(x) = \int \tilde P_{F}(x,z) \tu(z)\,dz
  \approx u(x)$, where $u(x)= \int P_{F}(x,z) \tu(z)\,dz$.
Similarly $\tilde{v}(x) \approx v(x)=\int P_{F}(x,z) \tilde{v}(z)\,dz$. Since $\tilde{u}(x_0)
  = \tilde{v}(x_0)$, we have $u(x_0)\approx v(x_0)$.
    By \cite[Theorem 2.18]{KSV2014}, ${u}(x) \approx {v}(x)$, provided $x\in D\cap B(z,r/8)$.
We use Lemma~\ref{HIforL} for the full range $x\in D\cap B(z,pr)$.
\end{proof}

Now, we have all the tools necessary to prove the main result of our paper. Since in the proof we follow the idea from \cite{MR2892584}, we only give its basic steps (for details see  \cite[Proof of Theorem 1]{MR2892584}).
\begin{proof}[Proof of Theorem \ref{Theorem1}]
By (\ref{eq:wp}) and (\ref{eq:estGradGreena}), we have the estimate
\begin{equation}\label{eq:0}
  \tG_D(x,y) \le G_D(x,y) + \Cg \int_{D} \frac{\tG_D(x,z)G_D(z,y)}{\delta_D(z)\land |y-z|} |b(z)|\,dz\,,\quad
x,y\in D\,.
\end{equation}
We consider  $\eta<1$, say $\eta=1/2$.
By Lemma~\ref{GDUnif} and the uniform integrability in Lemma~\ref{lem:boundKappa} (see  \eqref{eq:ineqGGGUU}) there is a constant $r>0$ so small that
\begin{align}
\int_{D_r} \frac{G_D(z,y)}{\deltaDD(z)\land |y-z|} |b(z)|\,dz  &<\frac{\eta}{\Cg}\,,\quad
  y \in D\,, \label{eq:1}\\
  \int_{D_r} \frac{G_D(x,z)G_D(z,y)}{G_D(x,y)(\deltaDD(z)\land |y-z|)} |b(z)|\,dz &<\frac{\eta}{\Cg}\,, \quad y \in D\,.\label{eq:2}
\end{align}
Here, $D_r = \{z \in D \colon \deltaDD(z) \leq r\}$.
We denote $$\rho=[\varepsilon \land r_0(D)\land r]/16\,,$$ with $\varepsilon=\varepsilon(d,\alpha,b,2/\kappa)$ of Lemma~\ref{Theorem1s}, see also Lemma~\ref{l:loc}.

To prove (\ref{eq:egf}) we will consider $x$ and $y$ in a partitions of $D\times D$.

\noindent First, we consider $y$ far from the boundary of $D$, say $\delta_D(y) \ge \rho/4$.
\begin{itemize}
	\item For $|x-y| \le \rho/8$, $G_D(x,y) \approx G_B(x,y) \approx U(x-y) \approx \tG_D(x,y)$ (we use Lemmas \ref{GreenEst1}, \ref{Theorem1s}, \ref{lem:GLDupper}).
	\item If $\rho/8<\delta_D(x)$ we use Harnack inequalities for $\mathcal{L}$ and $\tilde{\mathcal{L}}$.
	\item For $\delta_D(x) < \rho/8$ we use Boundary Harnack principles (see Lemma~\ref{BHPforL}, \cite[Theorem 2.18]{KSV2014}).
\end{itemize}

\noindent Next, suppose that $\deltaDD(y) \le \rho /4$. Here, the difficulty lies in the fact $\tGDD$ is non-symmetric.\\
In the proof of lower bounds we consider two cases: $x$ close to $y$ and $x$ far away from $y$.
\begin{itemize}
\item In the case $|x-y| \le \rho$, we locally approximate $D$ by the small $C^{1,1}$ set $F$ such that $\delta_D(x) = \delta_F(x)$ and  $\delta_D(x) = \delta_F(x)$ (see \cite[Lemma 1]{MR2892584}). Then $\tG_D(x,y) \ge \tG_F(x,y) \approx G_F(x,y) \approx G_D(x,y)$ (see Lemma \ref{GreenEst1}).\item For $|x-y| > \rho$ and $\delta_D(x) \ge \rho/4$ we use Harnack inequalities. For $\delta_D(x) \le \rho/4$ we use boundary Harnack principles.
\end{itemize}

\noindent In the next step, we prove the upper bound in (\ref{eq:egf}) for $\deltaDD(x) \ge \rho/4$.
We have already proved that for $z \in D \setminus D^r$,
$$
c_1^{-1} G_D(x,z) \le \tG_D(x,z) \le c_1 G_D(x,z)\,.
$$
By (\ref{def:kappa}), Lemma~\ref{lem:boundKappa}, Lemma~\ref{lem:GLDupper}, (\ref{eq:0}) and (\ref{eq:1}),
\begin{align}
  \tG_D(x,y)  & \le A G_D(x,y) + \Cg\int_{D_r} \frac{\tG_D(x,z)G_D(z,y)}{|y-z| \land
    \deltaDD(z)} |b(z)|\,dz\,, \label{eq:toi} \\
  & \le AG_D(x,y) + B(x)\,, \label{eq:3}
\end{align}
where $A = 1+c_1 \Cg\CIII$ and $B(x)=\eta \CXXI U(\delta_{D^r}(x))$.
Now, plugging (\ref{eq:3}) into
(\ref{eq:toi}), and using (\ref{eq:1}), (\ref{eq:2}) and induction, we get for $n=0,1,\ldots$,
\begin{equation}\label{eq:AB}
\tG_D(x,y) \le A\big(1 + \eta +\cdots + \eta^n \big)G_D(x,y) +
\eta^n B(x)\,.
\end{equation}
In consequence,
\begin{equation}\label{eq:ubi}
\tG_D(x,y) \le \frac{A}{1- \eta} G_D(x,y)\,.
\end{equation}
Finally, we prove the upper bound in (\ref{eq:egf}) when $\deltaDD(x)< \rho/4$.
\begin{itemize}
\item If $|x-y| > \rho$, we use boundary Harnack principles.
\item For $|x-y| \le \rho$, consider the same set $F$ as above. We have
\end{itemize}

$$
  \tGDD(x,y) =
   \tG_F(x,y) + \int_{D \setminus F} \tP_F(x,z) \tGDD(z,y) \,dz\,.
$$
By Lemma~\ref{Theorem1s} and (\ref{PoissonComp}), $\tG_F(x,y) \approx G_F(x,y)$ and $\tP_F(x,z)\approx P_F(x,z)$. We already know that  for $|z-y| > \rho$, $\tGDD(z,y) \approx G(z,y)$. Thus,
$$
\tGDD(x,y) \approx G_F(x,y) + \int_{D \setminus F} P_F(x,z)
 G_D(z,y) \,dz = G_D(x,y)\,.
$$
The proof of Theorem~\ref{Theorem1} is complete.
\end{proof}



\end{document}